\documentclass{amsart}
\usepackage{amssymb,latexsym}
\usepackage{amsmath}
\usepackage[dvipdfm]{graphicx}
 \usepackage{color}
\usepackage{enumerate}
\newenvironment{enumeratei}{\begin{enumerate}[\upshape (i)]}{\end{enumerate}}
\numberwithin{equation}{section}
\theoremstyle{plain}
 \newtheorem{theorem}{Theorem}[section]
 \newtheorem{lemma}[theorem]{Lemma}

\theoremstyle{definition}
 \newtheorem{definition}[theorem]{Definition}

 \newtheorem{example}[theorem]{Example}
 
 \newtheorem{case}[theorem]{Case}
 \newtheorem*{initstep}{Initial step}
 \newtheorem*{successorstep}{Successor  step}
 \newtheorem*{limitstep}{Limit step}

\newcommand \url [1] {\tt{#1}}
\newcommand \labqa {\textup{(C1)}}
\newcommand \labqb {\textup{(C2)}}
\newcommand \labaxa {\textup{(A1)}}
\newcommand \labaxh {\textup{(A2)}}
\newcommand \labaxb {\textup{(A3)}}
\newcommand \labaxc {\textup{(A4)}}
\newcommand \labaxd {\textup{(A5)}}
\newcommand \labaxe {\textup{(A6)}}
\newcommand \labaxf {\textup{(A7)}}
\newcommand \labaxg {\textup{(A8)}}

\newcommand \princ[1] {\textup{Princ}(#1)}
\newcommand \cg[2] {\textup{con}(#1,#2)}
\newcommand \congen[1] {\textup{con}(#1)}
\newcommand \conbiggen[1] {\textup{con}\bigl(#1\bigr)}
\newcommand \cgi[3] {\textup{con}_{#1}(#2,#3)}

\newcommand\alg [1] {{\mathcal #1}}
\newcommand \leqnu {\mathrel{\leq_\nu}}
\newcommand \lessnu {\mathrel{<_\nu}}
\newcommand \leqnuspa {\mathrel{\leq_{\spa\nu}}}

\newcommand \parallelnu {\mathrel{\parallel_\nu}}
\newcommand \pairs [1] {{\textup{Pairs}^{\leq}(#1)}}
\newcommand \Quord[1]{\textup{Quord}(#1)}
\newcommand \quo[2] {\textup{quo}(#1,#2)}
\newcommand \quos[1] {\textup{quo}(#1)}
\newcommand \bigquos[1] {\textup{quo}\bigl(#1\bigr)}
\newcommand \nablaell [1] {\nabla_{\kern -2pt #1}}
\newcommand \blokk[2] {[#1]#2}
\newcommand \vcs[1]{{#1^{\scriptscriptstyle{\mathord{\pmb{\vartriangle}}}}}}
\newcommand \fcs[1]{{#1^\ast}}
\newcommand \acs[1]{{#1_\ast}}
\newcommand \spa[1]{{#1^{\scriptscriptstyle{{{\mathord{ \blacktriangleright  }}}}}}}
\newcommand \vesz[1]{{#1}^{\bullet}}
\newcommand \grlat {L_{\textup{GG}}} 
\newcommand \restrict[2] {{#1\rceil_{#2}}}
\newcommand \qq[1] {{\kern #1 pt}}

\newcommand\ideal[1]{\mathord\downarrow #1}
\newcommand\filter[1]{\mathord\uparrow #1}
\newcommand \tuple [1] {\langle #1\rangle}
\newcommand \pair [2] {\tuple{#1,#2}}
\renewcommand\phi{\varphi}
\renewcommand\emptyset{\varnothing}

\newcommand \tbf [1] {\textbf{#1}} 
\newcommand \set[1] {\{#1\}}
\newcommand \bigset[1] {\bigl\{#1\bigr\}}
\renewcommand\phi{\varphi}
\renewcommand\epsilon{\varepsilon}
 
\newcommand \Then {\mathrel{\Longrightarrow}} 
\newcommand \nonparallel {\mathrel{
\not\mathord{\kern -1.5 pt\parallel}}}

\newcommand\init [1] {} 
\newcommand\nothing [1] {}

\begin{document}
\title[Principal congruences of a countable lattice]
{The ordered set of principal congruences of a countable lattice}

\author[G.\ Cz\'edli]{G\'abor Cz\'edli}
\email{czedli@math.u-szeged.hu}
\urladdr{http://www.math.u-szeged.hu/$\sim$czedli/}
\address{University of Szeged, Bolyai Institute. 
Szeged, Aradi v\'ertan\'uk tere 1, HUNGARY 6720}

\thanks{This research was supported by
the European Union and co-funded by the European Social Fund  under the project ``Telemedicine-focused research activities on the field of Mathematics, 
Informatics and Medical sciences'' of project number    ``T\'AMOP-4.2.2.A-11/1/KONV-2012-0073'', and by  NFSR of Hungary (OTKA), grant number 
K83219}

\dedicatory{To the memory of Andr\'as P.\ Huhn}

%
%


\date{May 7, 2013}

\begin{abstract} For a lattice $L$, let $\princ L$ denote the ordered set of principal congruences of $L$. In a pioneering paper, G.\ Gr\"atzer characterized the ordered sets $\princ L$ of finite lattices $L$; here we do the same for countable lattices.
He also showed that each bounded ordered set $H$ is isomorphic to $\princ L$ of a bounded lattice $L$. We prove a related statement: if an ordered set $H$ with least element is the union of a chain of principal ideals, then $H$ is isomorphic to $\princ L$ of some lattice $L$.  
\end{abstract}

\maketitle

\section{Introduction}\label{introsection}
\subsection{Historical background}
A classical theorem of \init{R.\,P.\ }Dilworth~\cite{dilwcollect} states  that each finite distributive lattice is isomorphic to the congruence lattice of a finite lattice. 
Since this first result, the \emph{congruence lattice representation problem} has attracted many researchers, and dozens of papers belonging to this topic  have been written. The story of this problem were mile-stoned by \init{A.\,P.\ }Huhn~\cite{huhn} and \init{E.\,T.\ }Schmidt~\cite{schmidtidnl}, reached its summit in
\init{F.\ }Wehrung~\cite{wehrung} and  \init{P.\ }R\r{u}\v{z}i\v{c}ka~\cite{ruzicka}, and was summarized in \init{G.\ }Gr\"atzer~\cite{grbypict}; see also  \init{G.\ }Cz\'edli~\cite{czgrepres} for some additional, recent references.

In \cite{ggprincl}, \init{G.\ }Gr\"atzer started an analogous new topic of Lattice Theory.  Namely, for a lattice $L$, let $\princ L=\pair{\princ L}{\subseteq}$ denote the   ordered set  of principal congruences of $L$. A congruence is \emph{principal} if it is generated by a pair $\pair ab$ of elements.
Ordered sets and lattices with 0 and 1 are called \emph{bounded}. Clearly, if $L$ is a bounded lattice, then $\princ L$ is a bounded ordered set. 
The pioneering theorem in \init{G.\ }Gr\"atzer~\cite{ggprincl} states the converse: each bounded  ordered set $P$ is isomorphic to $\princ L$ for an appropriate bounded lattice $L$. 
Actually, the lattice he constructed is  of length 5.
Up to isomorphism, he also characterized finite bounded ordered sets as the $\princ L$ of finite lattices $L$. 

\subsection{Terminology} Unless otherwise stated, we follow the standard terminology and notation of Lattice Theory; see, for example, \init{G.\ }Gr\"atzer~\cite{GGLT}. 
Our terminology for  weak perspectivity is the classical one taken from \init{G.\ }Gr\"atzer~\cite{grglt}. 
\emph{Ordered sets} are nonempty sets equipped with \emph{orderings}, that is, with reflexive, transitive, antisymmetric relations. Note that an 
ordered set is often called a \emph{partially ordered set}, which is a rather long expression, or a \emph{poset}, which is not tolerated by spell-checkers, or an \emph{order}, which has several additional meanings.

\subsection{Our result}
Motivated by \init{G.\ }Gr\"atzer's theorem mentioned above, our goal is to prove the following theorem. A set $X$ is \emph{countable} if it is finite or countably infinite, that is, if $|X|\leq \aleph_0$. An ordered set $P$ is \emph{directed} if each two-element subset of $P$ has an upper bound in $P$. Nonempty down-sets of $P$ and subsets $\ideal c=\set{x\in P: x\leq c}$   are called \emph{order ideals} and \emph{principal $($order$)$ ideals}, respectively.

\begin{theorem}\label{thmmain}\
\begin{enumeratei}
\item\label{thmmaina} An ordered set $P=\tuple{P;\leq}$ is isomorphic to $\princ L$ for some \emph{countable}  lattice $L$ if and only if  $P$ is a countable directed ordered set with zero.
\item\label{thmmainb} If $P$ is an ordered set with zero and  it is the union of a chain of principal ideals, then there exists a lattice $L$ such that $P\cong \princ L$.
\end{enumeratei}
\end{theorem}

An alternative way of formulating the condition in part \eqref{thmmainb} is to say that $0\in P$ and there is a  cofinal chain in $P$. 
For a pair $\pair ab\in L^2$ of elements, the least congruence collapsing $a$ and $b$ is denoted by $\cg ab$ or $\cgi Lab$. 
As it was pointed out in \init{G.\ }Gr\"atzer~\cite{ggprincl}, the rule 
\begin{equation}
\cg{a_i}{b_i}\subseteq 
\cg{a_1\wedge b_1\wedge a_2\wedge b_2}{a_1\vee b_1\vee a_2\vee b_2}\,
\text{ for }i\in\set{1,2}
\label{prdirT}
\end{equation}
implies that $\princ L$ is always a directed ordered set with zero. Therefore, the first part of the theorem will easily be concluded from the second one. 
To compare part \eqref{thmmainb} of our theorem to \init{G.\ }Gr\"atzer's result, note 
that a  bounded ordered set $P$ 
is always a union of a (one-element) chain of principal  ideals. Of course, no \emph{bounded} lattice $L$ can represent $P$ by $P\cong\princ L$ if $P$ has no greatest element. 

\subsection{Method}
First of all, we need the key idea, illustrated by Figure~\ref{fig4}, from \init{G.\ }Gr\"atzer~\cite{ggprincl}.

Second, we feel that without the quasi-coloring technique developed in \init{G.\ }Cz\'edli~\cite{czgrepres}, the investigations leading to this paper would have not even begun. 
As opposed to colorings, the advantage of quasi-colorings is that we have joins (equivalently, the possibility of generation) in their range sets. 
This allows us to decompose our construction into  a sequence of elementary steps. Each step is accompanied by a quasiordering. If several steps, possibly infinitely many steps, are carried out, then the join
of the corresponding quasiorderings gives a satisfactory insight into the construction.
Even if it is  the ``coloring versions'' of some lemmas that we only use at the end, it is worth allowing their quasi-coloring versions since this way the proofs are simpler and the lemmas become more general.

Third, the idea of using appropriate auxiliary structures is taken from \init{G.\ }Cz\'edli~\cite{112gen}. Their role is to accumulate all the assumptions our induction steps will need.

\section{Auxiliary statements and structures}
The rest of the paper is devoted to the proof of Theorem~\ref{thmmain}.
\subsection{Quasi-colorings and auxiliary structures}
A \emph{quasiordered set} is a structure $\tuple{H;\nu}$ where $H\neq \emptyset$ is a set and $\nu\subseteq H^2$ is a reflexive, transitive relation on $H$. Quasiordered sets are also called preordered ones. Instead of $\pair x y\in \nu$, we usually write $x \leqnu y$. Also,
we write  $x\lessnu y$ and  $x\parallelnu y$ for the conjunction 
of $x\leqnu y$ and $y\not\leqnu x$, and that of $\pair x y\notin\nu$ and $\pair yx\notin \nu$, respectively. If $g\in H$ and $x\leqnu g$ for all $x\in H$, then $g$ is a \emph{greatest element} of $H$; \emph{least elements} are defined dually. They are not necessarily unique; if they are, then they are denoted by $1_H$ and $0_H$.
If for all $x,y\in H$, there exists a $z\in H$ such that $x\leqnu z$ and $y\leqnu z$, then $\tuple{H;\nu}$ is a \emph{directed} quasiordered set.
Given $H\neq \emptyset$, the set of all quasiorderings on $H$ is denoted by $\Quord H$. It is a complete lattice with respect to set inclusion. For $X\subseteq H^2$, the least quasiorder on $H$ that includes $X$ is denotes by $\quos X$. We write $\quo xy$ instead of $\quos{\set{\pair ab}}$.

\begin{figure}[ht]
\centerline
{\includegraphics[scale=1.0]{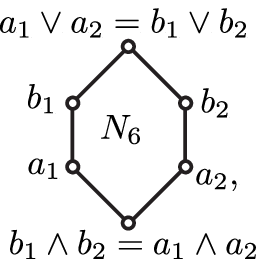}}
\caption{The lattice $N_6$ \label{fig1}}
\end{figure}

Let $L$ be a lattice. For $x,y\in L$, $\pair xy$ is called an \emph{ordered pair} of $L$ if $x\leq y$.  The set of ordered pairs of $L$ is denoted by $\pairs L$. Note that we shall often use that $\pairs S\subseteq \pairs L$ holds for sublattices $S$ of $L$; this explains why we work with ordered pairs rather than intervals. Note also that $\pair ab$ is an ordered pair if{f} $b/a$ is a quotient; however, the concept of ordered pairs fits better to previous work with quasi-colorings.

By a \emph{quasi-colored lattice} we mean a structure $\alg L=\tuple{L;\gamma,H,\nu}$ where $L$ is a lattice, $\tuple{H;\nu}$ is a quasiordered set, $\gamma\colon \pairs L\to H$ is a surjective map, and for all $\pair{u_1}{v_1},\pair{u_2}{v_2}\in \pairs L$, 
\begin{itemize}
\item[\labqa]if $\gamma(\pair{u_1}{v_1})\leqnu \gamma(\pair{u_2}{v_2})$, then $\cg{u_1}{v_1}\leq \cg{u_2}{v_2}$; 
\item[\labqb]if $\cg{u_1}{v_1}\leq \cg{u_2}{v_2}$, then $\gamma(\pair{u_1}{v_1})\leqnu \gamma(\pair{u_2}{v_2})$.
\end{itemize}
This concept is taken from \init{G.\ }Cz\'edli~\cite{czgrepres}. Prior to \cite{czgrepres},
the name ``coloring" was used for surjective maps onto antichains satisfying \labqb{} in \init{G.\ }Gr\"atzer, \init{H.\ }Lakser, and \init{E.T.\ }Schmidt~\cite{grlaksersch}, and for surjective maps onto antichains satisfying \labqa{} in \init{G.\ }Gr\"atzer~\cite[page
39]{grbypict}. However, in \cite{czgrepres}, \cite{grlaksersch}, and  \cite{grbypict}, $\gamma(\pair uv)$ was defined only for covering pairs $u\prec v$. 
To emphasize that $\cg{u_1}{v_1}$ and $\cg{u_2}{v_2}$ belong to the ordered set $\princ L$, we usually write $\cg{u_1}{v_1}\leq \cg{u_2}{v_2}$ rather than $\cg{u_1}{v_1}\subseteq \cg{u_2}{v_2}$.
It follows easily from \labqa{}, \labqb{}, and the surjectivity of $\gamma$ that if $\tuple{L;\gamma,H,\nu}$ is a quasi-colored set, then 
$\tuple{H;\nu}$ is a directed quasiordered set with least element; possibly with many least elements.

We say that a quadruple $\tuple{a_1,b_1,a_2,b_2}\in L^4$ is an \emph{$N_6$-quadruple} of $L$ if 
\[\set{b_1\wedge b_2=a_1\wedge a_2,\,\, a_1<b_1,\,\,a_2<b_2,\,\, a_1\vee a_2=b_1\vee b_2}
\] 
is a six-element sublattice, see Figure~\ref{fig1}. If, in addition, $b_1\wedge b_2=0_L$ and $a_1\vee a_2=1_L$, then we speak of a \emph{spanning $N_6$-quadruple}. 
An $N_6$-quadruple of $L$ is called a \emph{strong $N_6$-quadruple} if it is a spanning one  and, for all $i\in\set{1,2}$ and $x\in L$, 
\begin{align}
0_L < x \leq b_i &\Then x\vee a_{3-i}=1_L, \text{ and} \label{labsa}\\
1_L>x \geq a_i&\Then x\wedge b_{3-i}=0_L\text.\label{labsb}
%
\end{align}
For a subset $X$ of $L^2$, the least lattice congruence including $X$ is denoted by $\congen X$. In particular, $\congen{\set{\pair ab}}=\cg ab$. The least and the largest congruence of $L$ are denoted by $\Delta_L$ and $\nablaell L$, respectively.

Now, we are in the position to define the key concept we need. 
In the present paper, by a \emph{auxiliary structure} we mean a structure
\begin{equation}
\alg L=\tuple{L;\gamma, H,\nu,\delta ,\epsilon }\label{auxstr}
\end{equation}
such that the following eight properties hold:
\begin{itemize}
\item[\labaxa] $\tuple{L;\gamma, H,\nu}$ is a quasi-colored lattice;
\item[\labaxh] the quasiordered set $\tuple{H;\nu}$ has exactly one least element, $0_H$, and at most one greatest element;
\item[\labaxb] $\delta $ and $\epsilon $ are $H\to L$ maps such that $\delta (0_H)=\epsilon (0_H)$ and, for all $x\in H\setminus\set{0_H}$, $\delta (x)\prec\epsilon (x)$; note that we often write $a_x$ and $b_x$ instead of $\delta (x)$ and $\epsilon (x)$, respectively;
\item[\labaxc] for all $p\in H$, $\gamma(\pair{\delta (p)}{\epsilon (p)})=p$;
\item[\labaxd] if $p$ and $q$ are distinct elements of $H\setminus\set{0_H}$, then $\tuple{\delta (p),\epsilon (p), \delta (q),\epsilon (q)}$ is an $N_6$-quadruple of $L$;
\item[\labaxe] if $p,q\in H$, $p\parallelnu q$, and  $\tuple{\delta (p),\epsilon (p), \delta (q),\epsilon (q)}$ is a spanning $N_6$-quadruple, then it is a strong $N_6$-quadruple of $L$;
\item[\labaxf] If $L$ is a bounded lattice and $|L|>1$, then 
\begin{align*} 
\bigl|\bigl\{x\in L:{} &0_L\prec x\prec 1_L\text{ and, for all elements }y \text{ in }\cr
&L\setminus\{0_L,1_L,x\},\,\, x\text{ is a complement of }y\bigr\}\bigr|\geq 3;
\end{align*}
\item[\labaxg] if $1_H\in H$ and  $|L|>1$, then  $\conbiggen{ \set{\pair{\delta (p)} {\epsilon (p)}: p\in H\text{ and } p\neq 1_H }} \neq \nablaell L$.
\end{itemize}
It follows from \labaxd{} that $\set{\delta(x),\epsilon(x)}=\set {a_x, b_x}$ is disjoint from $\set{0_L,1_L}=\emptyset$, provided $|H|\geq 3$ and $x\in H\setminus\set{0_H}$. 

If $\tuple{H;\nu}$ is a quasiordered set, then $\Theta_\nu=\nu\cap\nu^{-1}$ is an equivalence relation, and the definition $\blokk x{\Theta_\nu}\leq \blokk y{\Theta_\nu}\iff x\leqnu y$  turns the quotient set $H/\Theta_\nu$ into an ordered set $\tuple{H/\Theta_\nu;\leq}$. The importance of our auxiliary structures is first shown by the following lemma.

\begin{lemma}\label{impclM} If $\alg L$ in \eqref{auxstr} is an auxiliary structure, then
the ordered set $\princ L$ is isomorphic to  $\tuple{H/\Theta_\nu;\leq}$. In particular, if $\nu$ is an  ordering, then $\princ L$ is isomorphic to the ordered set $\tuple{H;\nu}$.
\end{lemma}
 
\begin{proof}
Clearly, $\princ L=\set{\cg xy: \pair xy\in\pairs L}$. 
Consider the map $\phi\colon \princ L\to H/\Theta_\nu$, defined by $\cg xy\mapsto \blokk{\gamma(\pair xy)}{\Theta_\nu}$.
If $\cg {x_1}{y_1}=\cg {x_2}{y_2}$, then  $\blokk{\gamma(\pair {x_1}{y_1})}{\Theta_\nu} = \blokk{\gamma(\pair {x_2}{y_2})}{\Theta_\nu}$ follows from \labqb{}. Hence, $\phi$ is a map. It is surjective since so is $\gamma$. Finally, it is bijective and an order isomorphism by \labqa{} and \labqb{}. 
\end{proof}

We say that an auxiliary structure $\alg L=\tuple{L;\gamma, H,\nu,\delta ,\epsilon }$ is \emph{countable} if $|L|\leq\aleph_0$ and $|H|\leq\aleph_0$.
Next, we give an example.

\begin{example}\label{exegy} Let $H$ be a set, finite or infinite, such that $0_H,1_H\in H$ and $|H|\geq 3$. Let us define  
$\nu=\bigquos{(\set{0_H}\times H) \cup (H\times \set{1_H})}$; note that $\tuple{H;\nu}$ is an ordered set (actually, a modular lattice of length 2). Let $L$ be the lattice depicted in Figure~\ref{fig2}, where
$\set{h,g,p,q,\dots}$ is the  set
 $H\setminus\set{0_H,1_H}$. For $x\prec y$, $\gamma(\pair xy)$ is defined by the labeling of edges. Note that, in Figure~\ref{fig2}, we often write $0$ and $1$ rather than $0_H$ and $1_H$, because of space consideration. Let $\gamma(\pair xx)=0_H$ for $x\in L$, and let  $\gamma(\pair xy)=1_H$ for $x<y$ if $x\not\prec y$. Let $\delta (0_H)=\epsilon (0_H)=x_0$. For $s\in H\setminus\set{0_H}$, we define $\delta (s)=a_s$ and $\epsilon (s)=b_s$. 
Now, obviously, $\alg L=\tuple{L;\gamma, H,\nu,\delta ,\epsilon }$ is an auxiliary structure. If $|H|\leq \aleph_0$, then $\alg L$ is countable.
\end{example}

\begin{figure}[ht]
\centerline
{\includegraphics[scale=1.0]{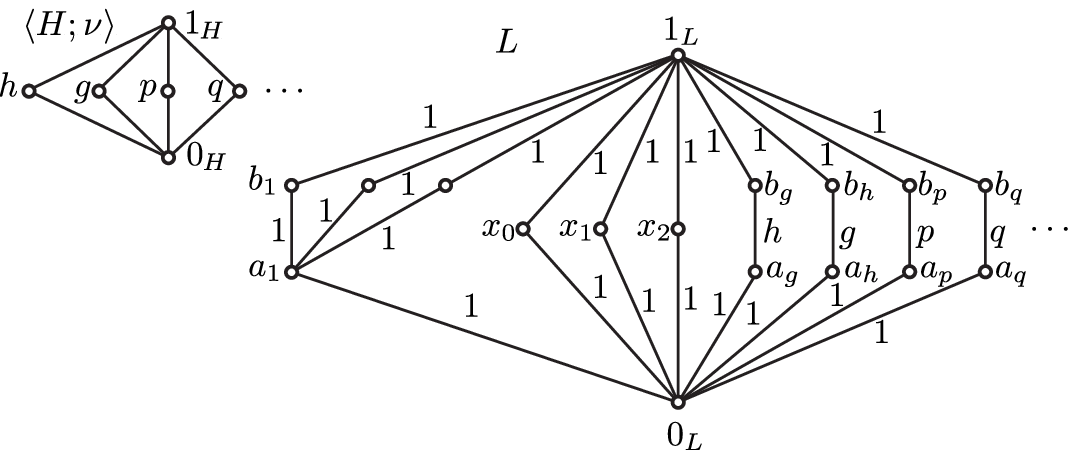}}
\caption{The auxiliary structure in Example~\ref{exegy} \label{fig2}}
\end{figure}

Substructures are defined in the natural way;  note that $\nu=\nu'\cap H^2$ will not be required below. Namely,

\begin{definition}\label{defbeagy}
Let 
$\alg L=\tuple{L;\gamma, H,\nu,\delta ,\epsilon }$ and $\alg L'=\tuple{L';\gamma', H',\nu',\delta ',\epsilon '}$ be auxiliary structures. We say that $\alg L$ is a \emph{substructure} of $\alg L'$ if the following hold:
\begin{enumeratei}
\item $L$ is a sublattice of $L'$, $H\subseteq H'$,  $\nu\subseteq\nu'$, and $0_{H'}=0_H$;
\item $\gamma$ is the restriction of $\gamma'$ to $\pairs L$,   $\delta$ is the restriction of $\delta'$ to $H$, and $\epsilon$ is the restriction of $\epsilon'$ to  $H$.
\end{enumeratei}
\end{definition}

Clearly, if $\alg L$,  $\alg L'$, and  $\alg L''$ are auxiliary structures such that $\alg L$ is a substructure of  $\alg L'$ and $\alg L'$ is a substructure of  $\alg L''$, then $\alg L$ is a substructure of  $\alg L''$; this fact will be used implicitly. 
The following lemma indicates how easily but efficiently we can work with auxiliary structures.

For an auxiliary structure $\alg L=\tuple{L;\gamma, H,\nu,\delta ,\epsilon }$ and an arbitrary (possibly empty) set $K$, we define the following objects. Let $\vcs H$ be the disjoint union $H\cup K\cup\set{1_{\vcs H}}$, and let $0_{\vcs H}=0_H$. Define $\vcs\nu\in\Quord{\vcs H}$ 
by \[\vcs\nu=\bigquos{\nu \cup(\set{0_{\vcs H}  } \times \vcs H )  \cup  (\vcs H\times\set{  1_{\vcs H}}) }\text. 
\]
Consider the lattice $\vcs L$ defined by Figure~\ref{fig3}, where $u,v,\dots$ denote the elements of $K$. The thick dotted lines indicate $\leq$ but not necessarily $\prec$; they are edges only if $L$ is bounded.
 Note that all  ``new'' lattice elements   distinct from $0_{\vcs L}$ and $1_{\vcs L}$, that is, all elements of $\vcs L\setminus(L\cup \set{0_{\vcs L}, 1_{\vcs L}})$, are complements of all ``old'' elements. Extend $\delta$ and $\epsilon$ to  maps $\vcs\delta,\vcs\epsilon\colon \vcs H \to{\vcs L}$ 
by letting  $\vcs{\delta }(w)=a_w$ and $\vcs{\epsilon }(w)=b_w$ for $w\in K\cup\set{1_{\vcs H}}$. Define $\vcs\gamma\colon \pairs{\vcs L}\to\vcs H$ by
\[
\vcs\gamma(
\pair xy)=
\begin{cases}
\gamma(\pair xy),& \text{if } \pair xy\in\pairs L,\cr
w, &\text{if } x=a_w,\,\, y=b_w,\text{ and }w\in K,\cr
0_{\vcs H},&\text{if }x=y,\cr
1_{\vcs H},&\text{otherwise.}
\end{cases}
\]
By space consideration again, the edge label $1$ in Figure~\ref{fig3} stands for $1_{\vcs H}$. 
Finally, let $\vcs{\alg L}=\tuple{\vcs L;\vcs\gamma,\vcs H,\vcs\nu, \vcs{\delta },\vcs{\epsilon }}$. The  straightforward proof of the following lemma will be omitted.

\begin{lemma}\label{lupstp} If $\alg L$ is an auxiliary structure, then so is $\vcs{\alg L}$. Furthermore, $\alg L$ is a substructure of $\vcs{\alg L}$, and if $\alg L$ and $K$ are countable, then so is $\vcs{\alg L}$.
Moreover, if $p,q\in \vcs H$ such that  $\set{p,q}\not\subseteq H$ and $p\parallel_{\vcs\nu}q$, then $\tuple{\vcs{\delta }(p), \vcs{\epsilon }(p),\vcs{\delta }(q), \vcs{\epsilon }(q)  }$ is a  strong $N_6$-quadruple.
\end{lemma}

Since new bottom and top elements are added, we say that $\vcs{\alg L}$ is obtained from $\alg L$ by a \emph{vertical extension}; this motivates the triangle aiming upwards in its notation.

\begin{figure}[ht]
\centerline
{\includegraphics[scale=1.0]{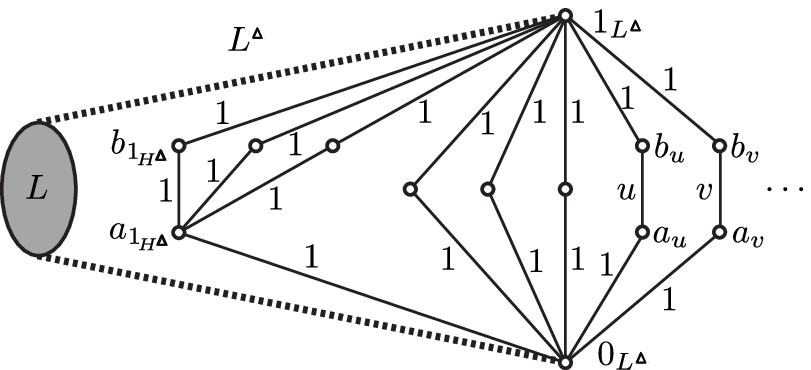}}
\caption{The auxiliary structure $\vcs{\alg L}$ \label{fig3}}
\end{figure}

\begin{figure}[ht]
\centerline
{\includegraphics[scale=1.0]{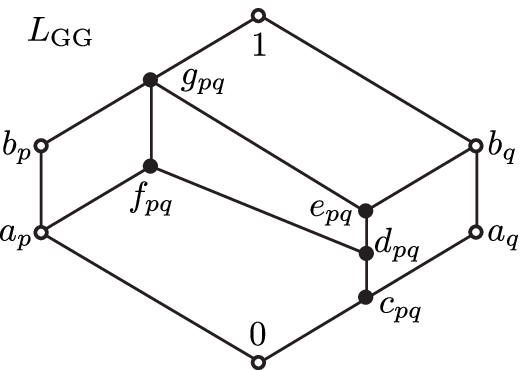}}
\caption{\init{G.\ }Gr\"atzer's lattice $\grlat$ \label{fig4}}
\end{figure}

\section{Horizontal extensions of auxiliary structures}
The key role in \init{G.\ }Gr\"atzer~\cite{ggprincl} is played by the lattice $\grlat$; see  Figure~\ref{fig4}. We also need this lattice. 
Assume that 
\begin{equation}
\begin{aligned}
\alg L&=\tuple{L;\gamma, H,\nu,\delta ,\epsilon }\text{ is an auxiliary structure, }p,q\in H\text{, } p\parallelnu q,\text{ and} \cr
&\tuple{a_p,b_p,a_q,b_q}=\tuple{\delta (p),\epsilon (p),\delta (q),\epsilon (q) } \text{ is a}\cr
&\text{spanning or, equivalently, a strong }N_6\text{-quadruple}.
\end{aligned}\label{asumLpar}
\end{equation}
The equivalence of ``spanning'' and ``strong'' in  \eqref{asumLpar} follows from \labaxe{}. We define a structure $\spa{\alg L}$ as follows, and it will take a lot of work to prove that it is an auxiliary structure. We call $\spa{\alg L}$ a \emph{horizontal extension} of $\alg L$; this explains the horizontal triangle in the notation.  
By changing the sublattice $\set{0_L,a_p,b_p,a_q,b_q,1_L}$ into an $\grlat$ as it is depicted in  Figure~\ref{fig4}, that is, by inserting the black-filled elements of Figure~\ref{fig4} into $L$, we obtain an ordered set denoted by $\spa L$; see also \eqref{spaorder} later for more exact details. (We will prove that $\spa L$ is a lattice and $L$ is a sublattice in it.) The construction of $\spa{\alg L}$ from  $\alg L$ is illustrated in  Figure~\ref{fig5}. Note that there can be much more elements and in a  more complicated way than indicated. 
The 
solid lines represent the covering relation but the 
dotted lines are not necessarily edges. The new lattice $\spa L$ is obtained from $L$ by inserting the black-filled elements. Note that while \init{G.\ }Gr\"atzer~\cite{ggprincl} constructed a lattice of length 5, here even the interval, say,  $[b_p, 1_L]$ can be of infinite length.

\begin{figure}[ht]
\centerline
{\includegraphics[scale=1.0]{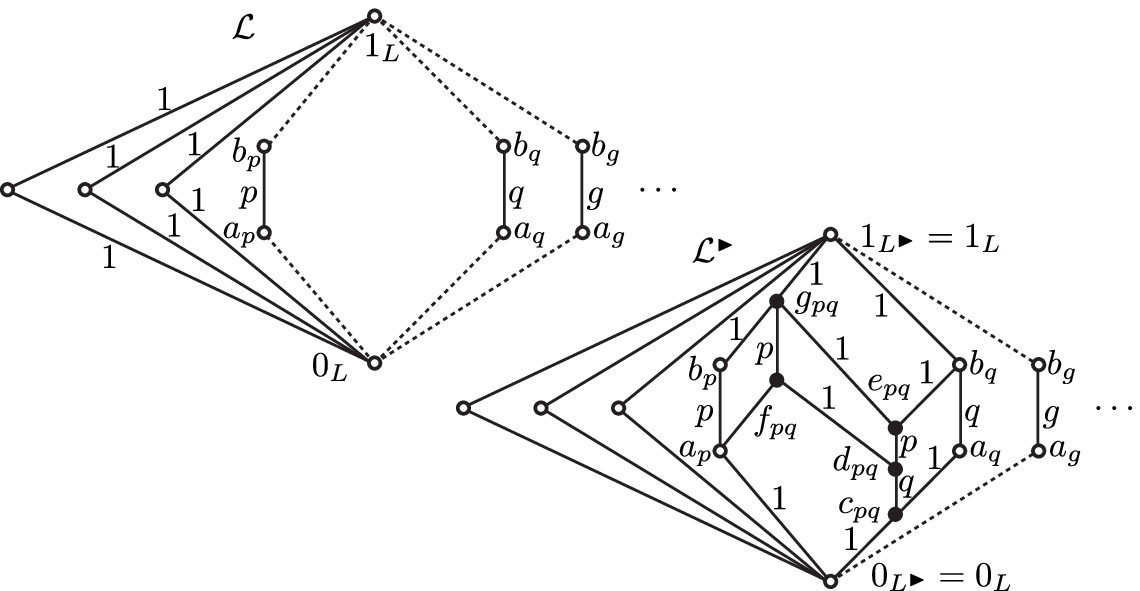}}
\caption{Obtaining $\spa{\alg L}$ from $\alg L$ \label{fig5}}
\end{figure}

Let $\spa H=H$. In $\Quord {\spa H}$, we define
$\spa\nu=\bigquos{\nu\cup\set{\pair pq}}$. We extend $\gamma$ to $\spa\gamma\colon \pairs{\spa L}\to \spa H$ by          
\[
\spa\gamma(\pair xy)=
\begin{cases}
\gamma(\pair xy),&\text{if }\pair xy\in\pairs L,\cr
p,&\text{if }\pair xy\in\set{\pair{d_{pq}}{e_{pq}}, \pair{f_{pq}}{g_{pq}}  },\cr
q,&\text{if }\pair xy\in\set{\pair{c_{pq}}{d_{pq}}, \pair{c_{pq}}{e_{pq}}  },\cr
0_{\spa H}, &\text{if } x=y,\cr
1_{\spa H}, &\text{otherwise.}
\end{cases}
\]
The definition of $\spa\gamma$ is also illustrated in Figure~\ref{fig5}, where the edge color $1$ stands for  $1_{\spa H}$. Finally, after letting $\spa{\delta }=\delta $ and $\spa{\epsilon }=\epsilon $, we define
\begin{equation} 
\spa{\alg L}=\tuple{\spa L;\spa\gamma, \spa H,\spa \nu,\spa{\delta },\spa{\epsilon }}\text.\label{spaldef}
\end{equation}

\begin{lemma}\label{spalislat}
If $\alg L$ satisfies \eqref{asumLpar}, then $\spa L$ is a lattice and $L$ is a sublattice of $\spa L$. 
\end{lemma}

\begin{proof} First, we describe the ordering of $\spa L$ more precisely; this description is the real  definition of $\spa L$. Let  
\begin{equation}
\begin{aligned}
N_6^{pq}&=\spa L\setminus L=\set{c_{pq},d_{pq},e_{pq},f_{pq},g_{pq}},\cr 
B_6^{pq}&=\set{0_L, a_p,b_p,a_q,b_q,1_L }\text{, and}\cr
S_6^{pq}&=\{0_L, a_p,b_p,a_q,b_q,c_{pq},d_{pq},e_{pq},f_{pq},g_{pq}, 1_L\} = N_6^{pq}\cup B_6^{pq}\text{.} \cr
\end{aligned}\label{sodiHGrj}
\end{equation}
Here $S_6^{pq}$ is isomorphic to the lattice $\grlat$, and its ``boundary'', $B_6^{pq}$, to $N_6$. The elements of $L$, $N_6^{pq}$, and $B_6^{pq}$ are called \emph{old}, \emph{new}, and \emph{boundary} elements, respectively. For $x,y\in \spa L$, we define $x\leq_{\spa L} y\iff$
\begin{equation}
\begin{cases}
x\leq_L y ,&\text{if }x,y\in L\text{, or}\cr
x\leq_{S_6^{pq}} y ,&\text{if }x,y\in S_6^{pq}\text{, or}\cr
\exists z\in  B_6^{pq}:
x\leq_L z\text{ and }z\leq_{S_6^{pq}}y,&\text{if }x\in L\setminus S_6^{pq}\text{ and } y\in N_6^{pq}\text{, or}\cr
\exists z\in  B_6^{pq}:
x\leq_{S_6^{pq}}z\text{ and }z\leq_L y,&\text{if }x\in N_6^{pq}\text{ and } y\in L\setminus S_6^{pq}\text.
\end{cases}
\label{spaorder}
\end{equation}
Observe that for $u_1,u_3\in B_6^{pq}$ and $u_2\in N_6^{pq}$, the conjunction of  $u_1\leq_{S_6^{pq}}u_2$ and $u_2 \leq_{S_6^{pq}} u_3$ implies
$\set{0_{\spa L},1_{\spa L}}\cap\set{u_1,u_3}\neq\emptyset$. Hence, it is straightforward to see that $\leq_{\spa L}$ is an ordering and $\leq_L$ is the restriction of $\leq_{\spa L}$ to $L$.

For $x\in N_6^{pq}$, there is a unique least element $\fcs x$ of $B_6^{pq}$ such that $x\leq_{S_6^{pq}}\fcs x$ (that is, $x\leq_{\spa L}\fcs x$). If $x\in L$, then we let $\fcs x=x$.  Similarly, for $x\in N_6^{pq}$, there is a unique largest element $\acs x$ of $B_6^{pq}$ such that 
$\acs x\leq_{S_6^{pq}} x$. Again, for $x\in L$, we let $\acs x=x$.  With this notation,  \eqref{spaorder} is clearly equivalent to
\begin{equation}
x\leq_{\spa L} y\iff
\begin{cases}
x\leq_L y,&\text{if }x,y\in L\text{, or}\cr
x\leq_{S_6^{pq}} y,&\text{if }x,y\in S_6^{pq}\text{, or}\cr
x\leq_L \acs y,&\text{if }x\in L\setminus S_6^{pq}\text{ and } y\in N_6^{pq}\text{, or}\cr
\fcs x\leq_L y,&\text{if }x\in N_6^{pq}\text{ and } y\in L\setminus S_6^{pq}\text.
\end{cases}\label{spanorder}
\end{equation}


Next, for  $x\parallel y\in\spa L$, we want to show that $x$ and $y$ has a join in $\spa L$. We can assume that $\set{x,y}$ has an upper bound $z$ in $N_6^{pq}$, because otherwise $\fcs x\vee_L \fcs y$ would clearly be the join of $x$ and $y$ in $\spa L$. If  $z$ belonged to $\set{c_{pq},d_{pq},e_{pq}}$, then the principal ideal $\ideal z$ (taken in $\spa L$) would be a chain, and this would  contradict $x\parallel y$. Hence, $z\in \set{f_{pq},g_{pq}}$. If both $x$ and $y$ belong to $N_6^{pq}$, then $x\parallel y$ gives $\set{x,y}=\set{e_{pq},f_{pq}}$, $z$ and $1_{\spa L}$ are the only upper bounds of $\set{x,y}$, and $z$ is the join of $x$ and $y$. Hence, we can assume that $x\in L$. 
If $y$ also belongs to $L$, then $x\leq\acs z$ and $y\leq \acs z$ yields 
$x\vee_L y\leq_{\spa L} \acs z\leq_{\spa L} z$, and $x\vee_L y$ is the join of $x$ and $y$ in $\spa L$ since $z$ was an arbitrary upper bound of $\set{x,y}$ in  $N_6^{pq}$. 

Therefore, we can assume that $x\in L$ and $y\in N_6^{pq}$. 
It follows from $b_p\wedge_L b_q=0_L$ that, for each $u\in L$, $\filter u\cap B_6^{pq}$ has a smallest element; we denote it by $\widehat u$. For $u\in N_6^{pq}$, we let $\widehat u=u$. 
Note that, for every $u\in \spa L$,  $\widehat u$ is the smallest element of $\filter u\cap S_6^{pq}$. The  existence of $z$, mentioned above, implies that $\widehat x\in  \set{a_p,b_p}$.

We assert that $\widehat x\vee_{S_6^{pq}} y= \widehat x\vee_{S_6^{pq}}\widehat y$ is the join of $x$ and $y$ in $\spa L$.   (Note that $\widehat x\vee_{S_6^{pq}} y\subseteq \set{f_{pq},g_{pq}}$.)
We can assume $y\in\set{c_{pq},d_{pq},e_{pq}}$ since otherwise $1_L$ is the only upper bound of $y$ in $L$ and  $x\vee_{\spa L}y=\widehat x\vee_{S_6^{pq}} y$ is clear. Consider an upper bound $t\in L$ of $x$ and $y$. 
Since $y\in\set{c_{pq},d_{pq},e_{pq}}$, we have $a_q\leq t$ and  $x\vee_L a_q\leq t$.  From  
$x\parallel y\in\spa L$ and $\widehat x\in  \set{a_p,b_p}$, we obtain $0_L<x\leq b_p$.
Since $\tuple{a_p,b_p,a_q,b_q}$ is a 
strong $N_6$-quadruple by \eqref{asumLpar}, 
the validity of \labaxe{} for $\alg L$ implies $\widehat x\vee_{S_6^{pq}} y\,\leq 1_{\spa L}=1_L=x\vee_L a_q\leq t$.
This shows that $\widehat x\vee_{S_6^{pq}} y$ is the join of $x$ and $y$ in $\spa L$. 
The case $x,y\in L$ showed that $\tuple{L;\vee}$ is a subsemilattice of $\tuple{\spa L;\vee}$.
For later reference, we summarize the description of join in a concise form as follows; note that $x\parallel y$ is not assumed here:
\begin{equation}
x\vee_{\spa L}y=\begin{cases}
\fcs x\vee_L \fcs y, &\text{if } \set{x,y}\not\subseteq \ideal{g_{pq}}\text{ or }\set{x,y}\subseteq L , \cr
\widehat x\vee_{S_6^{pq}} \widehat y&\text{otherwise, that is, if } \set{x,y}\subseteq \ideal{g_{pq}}\text{ and }\set{x,y}\not\subseteq L\text.
\end{cases}\label{joindeR}
\end{equation}

We have shown that any two elements of $\spa L$ have a join. Although $S_6^{pq}$ and the construction of $\spa L$ are not exactly selfdual, by interchanging the role of $\set{f_{pq},g_{pq}}$ and that of $\set{c_{pq},d_{pq},e_{pq}}$, we can easily dualize the argument above. Thus,  we conclude that $\spa L$ is a lattice and $L$ a sublattice of $\spa L$.
\end{proof}

The following lemma is due to \init{R.\,P.\ }Dilworth~\cite{dilworth1950a}, see also  \init{G. }Gr\"atzer~\cite[Theorem III.1.2]{grglt}.

\begin{lemma}\label{llajtorja} If $L$ is a  lattice and  $\pair{u_1}{v_1},\pair{u_2}{v_2}\in\pairs L$, then the following three conditions are equivalent.
\begin{enumeratei}
\item $\cg{u_1}{v_1}\leq \cg{u_2}{v_2}$; 
\item $\pair{u_1}{v_1}\in \cg{u_2}{v_2}$;
\item there exists an $n\in\mathbb N$ and there are $x_{i} \in L$ for $i\in\set{0,\dots,n}$ and $\pair{y_{i j}}{z_{i j}} \in\pairs L$ for 
$\pair ij\in\set{1,\dots,n}\times\set{0,\dots,n}$ such that the following equalities and inequalities hold:
\begin{equation}
\begin{aligned}
u_1&=x_{0}\leq x_{1}\leq\dots\leq x_{n-1}\leq x_{n}=v_1\cr
y_{i0} &=x_{i-1}\text{,  }y_{in}=u_2\text{,  }z_{i0}=x_i\text{, and }z_{in}=v_2\text{ for }1\leq i\leq n,\cr  
y_{i,j-1}&= z_{i,j-1}\wedge y_{ij}
\text{ and } z_{i,j-1}\leq z_{ij}
\text{ for }j \text{ odd, }  i,j\in\set{1,\dots,n},\cr
z_{i,j-1} & = y_{i,j-1}\vee z_{ij} \text{ and }y_{i,j-1}\geq y_{ij}\text{ for }j \text{ even, }  i,j\in\set{1,\dots,n}\text.
\end{aligned}\label{lajtorjaformula}
\end{equation}
\end{enumeratei}
\end{lemma}

The situation of  Lemma~\ref{llajtorja} is outlined in Figure~\ref{fig6}; note that not all elements are depicted, and the elements are not necessarily distinct. The second half of \eqref{lajtorjaformula} says that, in terms of \init{G.\ }Gr\"atzer~\cite{grglt}, $\pair{y_{i,j-1}}{z_{i,j-1}}$ is \emph{weakly} up or down perspective into $\pair{y_{ij}}{z_{ij}}$; up for $j$ odd and down for $j$ even. Besides weak perspectivity, we shall also need a more specific concept; recall that $\pair{x_1}{y_1}$ is \emph{perspective} to $\pair{x_2}{y_2}$ if there are $i,j\in\set{1,2}$ such that $i\neq j$, 
$x_i=y_i\wedge x_j$, and $y_j=x_j\vee y_i$.

\begin{figure}[ht]
\centerline
{\includegraphics[scale=1.0]{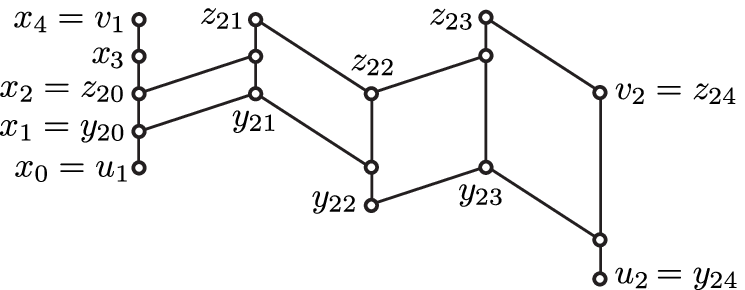}}
\caption{Illustrating Lemma~\ref{llajtorja} for $n=4$ \label{fig6}}
\end{figure}

For a quasiordered set $\tuple{H;\nu}$ and  $p, q_1,\dots,q_n\in H$, we say that $p$ is a \emph{join} of the elements $q_1,\dots,q_n$, in notation, $p=\bigvee_{i=1}^n q_i$, if $q_i\leqnu p$ for all $i$ and, for every $r\in H$, the conjunction of $q_i\leqnu r$ for $i=1,\dots,n$ implies $p\leqnu r$. This concept is used in the next lemma. Note that even if a join exists, it need not be unique. 

\begin{lemma}[``Chain Lemma'' for quasi-colored lattices]\label{chainlemma} If $\tuple{  L; \gamma,  H,\nu}$ is a quasi-colored lattice and $\set{u_0\leq u_1\leq\dots \leq u_n}$ is a finite chain in $L$, then 
\begin{equation}
\gamma(\pair{u_0}{u_n})=\bigvee_{i=1}^n \gamma(\pair{u_{i-1}}{u_i})\quad\text{ holds in }\tuple{H;\nu}\text.\label{lancjoin}
\end{equation}
\end{lemma}

\begin{proof} Let $p=\gamma(\pair{u_0}{u_n})$ and $q_i=\gamma(\pair{u_{i-1}}{u_i})$. Since $\cg{u_{i-1}}{u_i}\leq \cg{u_0}{u_n}$, \labqb{} yields $q_i\leqnu p$ for all $i$. Next, assume that $r\in H$ such that $q_i\leqnu r$ for all $i$. By the surjectivity of $\gamma$, there exists a $\pair vw\in\pairs L$ such that $\gamma(\pair vw)=r$. It follows by \labqa{} that $\pair{u_{i-1}}{u_i}\in \cg{u_{i-1}}{u_i}\leq \cg vw$. Since $\cg vw$ is transitive and collapses the pairs $\pair{u_{i-1}}{u_i}$, it collapses $\pair{u_{0}}{u_n}$. Hence, $\cg{u_{0}}{u_n}\leq \cg vw$, and \labqb{} implies $p\leqnu r$.
\end{proof}

Now, we are in the position to deal with the following lemma.

\begin{lemma}\label{mainlemma} The structure $\spa {\alg L}$, which is defined in \eqref{spaldef} with assumption \eqref{asumLpar}, is an auxiliary structure, and $\alg L$ is a substructure of $\spa {\alg L}$. Furthermore, if $\alg L$ is countable, then so is $\spa {\alg L}$.
\end{lemma}

\begin{proof} Since we work both in $\alg L$ and $\spa{\alg L}$, relations, operations and maps are often subscripted by the relevant structure. 
By Lemma~\ref{spalislat}, $\spa L$ is a lattice. Obviously, \labaxb{} and \labaxf{} hold for  $\spa{\alg L}$. Since $\spa\gamma$ is an extension of $\gamma$, $\spa{\delta}= \delta$, $\spa{\epsilon}= \epsilon$, and $L$ is a sublattice of $\spa L$,  we obtain that
 \labaxc{} and \labaxd{} hold in $\spa{\alg L}$.

Let $r_1,r_2\in\spa H$. 
Since $\nu$ is transitive, $p\parallelnu q$, and $\spa\nu=\bigquos{\nu\cup \set{\pair pq}}$, we obtain that 
\begin{equation}\label{twopssb}
\pair{r_1}{r_2}\in \spa\nu\iff
r_1\leqnu p\text{ and }q\leqnu r_2\text{, or }r_1\leqnu r_2\text.
\end{equation}
This clearly implies that \labaxh{} holds for $\spa{\alg L}$.

It follows from \labqa{} that if $\pair xy\in \pairs L$ and $\gamma(\pair xy)=1_{H}$, then we have $\cgi Lxy=\nablaell L$. Combining this with \labaxf{}, we obtain easily that for all $\pair xy\in\pairs {\spa L}$,
\begin{equation}
\spa\gamma(\pair xy)=1_{\spa H} \Then \cgi{\spa L}xy  = \nablaell {\spa L}\text.\label{haegyakok}
\end{equation}

Let $\Theta$ denote the congruence of $L$ described in \labaxg{}. Consider the equivalence relation $\spa\Theta$ on $\spa L$ whose classes (in other words, blocks) are the $\Theta$-classes, $\set{c_{pq},d_{pq},e_{pq}}$ and $\set{f_{pq},f_{pq}}$. Based on \eqref{joindeR} and its dual, a  straightforward argument shows that, for all $x,y\in\spa L$,  $\pair{x\wedge y}{x}\in\spa\Theta$ 
if{f}
$\pair y{x\vee y}\in \spa\Theta$. Clearly, the intersection of $\spa\Theta$ and the ordering of $\spa L$ is transitive. Hence, we conclude that $\spa\Theta$ is a congruence on $\spa L$. Since it is distinct from $\nablaell{\spa L}$, 
$\spa{\alg L}$ satisfies \labaxg{}.

Next, we prove the converse of \eqref{haegyakok}. Assume that $\pair xy\in\pairs{\spa L}$ such that $\spa\gamma(\pair xy)\neq 1_{\spa H}$; we want to show that 
$\cgi{\spa L}xy\neq\nablaell{\spa L}$. Since this is clear if $x=y$, we assume $x\neq y$. First, if $x,y\in L$, then let $r=\gamma(\pair xy)$.
Applying  \labqa{} to $\gamma$ and \labaxc{} to $\alg L$,  we obtain $ \cgi{L}xy =\cgi{L}{\delta (r)}{\delta (r)}$. Hence $\Theta$, which we used in the previous paragraph, collapses  $\pair xy$, and $\cgi{\spa L}xy\subseteq\spa\Theta\subset  \nablaell{\spa L}$.  Second, if $\set{x,y}\cap L=\emptyset$, then 
$\pair xy$ is perspective to $\pair{a_p}{b_p}$ or $\pair{a_q}{b_q}$, whence 
$ \cgi{\spa L} xy \in
\bigset{  \cgi{\spa L}{a_p}{b_p},  \cgi{\spa L}{a_q}{b_q}   }$ reduces the present case to the 
previous one. Finally, $|L\cap\set{x,y}|=1$ is excluded since then $\pair xy$ would be  $1_{\spa H}$-colored. Now, after  verifying the converse of \eqref{haegyakok}, we have proved 
that for all $\pair xy\in\pairs {\spa L}$,
\begin{equation}
\spa\gamma(\pair xy)=1_{\spa H} \iff \cgi{\spa L}xy  = \nablaell {\spa L}\text.\label{haiirkok}
\end{equation}

Next, to prove that $\spa\gamma$ satisfies \labqa{}, assume that $\pair{u_1}{v_1},\pair{u_2}{v_2}\in \pairs{\spa L}$ such that $\spa\gamma(\pair{u_1}{v_1} )\leqnuspa \spa\gamma(\pair{u_2}{v_2} )$.
Let $r_i=\spa\gamma(\pair{u_i}{v_i} )$, for $i\in\set{1,2}$. We have to show $\cgi{\spa L}{u_1}{v_1}\leq \cgi{\spa L}{u_2}{v_2}$. By \eqref{haiirkok}, we can assume that $r_2\neq 1_{\spa H}$.  Thus, by \labaxh{}, we have $r_1\neq 1_{\spa H}$. We can also assume that  $r_1\neq 0_{\spa H}$ since otherwise  $\cgi{\spa L}{u_1}{v_1}=\cgi{\spa L}{u_1}{u_1}=\Delta_{\spa L}$ 
would clearly imply $\cgi{\spa L}{u_1}{v_1}\leq \cgi{\spa L}{u_2}{v_2}$. Thus, $r_1,r_2\in H\setminus\set{0_H,1_H}$. By the construction of $\spa L$, $\pair {u_i}{v_i}$ is perspective to some $\pair {u_i'}{v_i'}\in\pairs L$ such that
$\spa\gamma(\pair{u_i}{v_i})=\spa\gamma(\pair{u_i'}{v_i}')$,  and perspectivity implies  $\cgi{\spa L}{u_i}{v_i}= \cgi{\spa L}{u_i'}{v_i'}$. Therefore, we can assume that $\pair{u_1}{v_1}, \pair{u_2}{v_2} \in \pairs L$, because  otherwise we could work with $\pair{u_1'}{v_1'}$ and $\pair{u_2'}{v_2'}$.

According  to  \eqref{twopssb}, we  distinguish two cases.
First, assume that $r_1\leqnu r_2$. Since $\spa\gamma$ extends $\gamma$, we have $\gamma(\pair{u_1}{v_1} ) = \spa\gamma(\pair{u_1}{v_1} )
=r_1\leqnu r_2=\spa\gamma(\pair{u_2}{v_2} )= \gamma(\pair{u_2}{v_2} )$. Applying \labqa{} to $\gamma$, we obtain $\pair{u_1}{v_1}\in \cgi{L}{u_1}{v_1}\leq \cgi{L}{u_2}{v_2}$. Using Lemma~\ref{llajtorja}, first in $L$ and then, backwards, in $\spa L$, we obtain $\pair{u_1}{v_1}\in\cgi{\spa L}{u_2}{v_2}$, which yields $\cgi{\spa L}{u_1}{v_1}\leq \cgi{\spa L}{u_2}{v_2}$.

Second, assume that $r_1\leqnu p$ and $q\leqnu r_2$. Since $\spa\gamma(\pair {a_p}{b_p})=\gamma(\pair {a_p}{b_p})=p$ and $\spa\gamma(\pair {a_q}{b_q})=q$ by \labaxc{}, the argument of the previous paragraph yields that we have
$\cgi{\spa L}{u_1}{v_1}\leq \cgi{\spa L}{a_p}{b_p}$
and $\cgi{\spa L}{a_q}{b_q}\leq \cgi{\spa L}{u_2}{v_2}$. 
Clearly (or applying Lemma~\ref{llajtorja} within $S_6^{pq})$, we have $\cgi{\spa L}{a_p}{b_p}\leq \cgi{\spa L}{a_q}{b_q}$. Hence, transitivity yields $\cgi{\spa L}{u_1}{v_1}\leq \cgi{\spa L}{u_2}{v_2}$. Consequently, $\spa\gamma$ satisfies  \labqa{}.

Next, to prove that $\spa\gamma$ satisfies \labqb{}, assume that $\pair{u_1}{v_1}, \pair{u_2}{v_2}\in\pairs{\spa L}$ such that $\cgi{\spa L}{u_1}{v_1} \leq \cgi{\spa L}{u_2}{v_2}$. Our purpose is to show the inequality $\spa\gamma(\pair{u_1}{v_1} )\leqnuspa \spa\gamma(\pair{u_2}{v_2} )$.  By \eqref{haiirkok}, we can assume $\cgi{\spa L}{u_2}{v_2}\neq \nablaell{\spa L}$, and we can obviously assume $u_1\neq v_1$. That is, $\set{  \cgi{\spa L}{u_1}{v_1},\, \cgi{\spa L}{u_2}{v_2}  }\cap \set{\Delta_{\spa L},\nablaell{\spa L}}=\emptyset$.
A pair $\pair {w_1}{w_2}\in\pairs{\spa L}$ is called \emph{mixed}
if $|\set{i: w_i\in L}|=1$. That is, if one of the components is old and the other one is new. It follows from the construction of $\spa{\alg L}$ and \eqref{haiirkok} that none of $\pair{u_1}{v_1}$ and $\pair{u_2}{v_2}$ is mixed. 
If $\pair{u_1}{v_1}$ is a new pair, that is, if $\set{u_1,v_1}\cap L=\emptyset$, then we can consider an old pair $\pair{u'_1}{v'_1}$ such that $\spa\gamma(\pair{u_1'}{v_1'} )=\spa\gamma(\pair{u_1}{v_1} )$ and, by perspectivity, $\cgi{\spa L}{u'_1}{v'_1} = \cgi{\spa L}{u_1}{v_1}$. Hence, we can assume that $\pair{u_1}{v_1}$ is an old pair, and similarly for the other pair. That is, we assume that  both $\pair{u_1}{v_1}$ and  $\pair{u_2}{v_2}$ belong to $\pairs{ L}$.

The starting assumption $\cgi{\spa L}{u_1}{v_1} \leq \cgi{\spa L}{u_2}{v_2}$ means that $\pair{u_1}{v_1} \in \cgi{\spa L}{u_2}{v_2}$. This is witnessed by Lemma~\ref{llajtorja}. Let $x_j, y_{ij}, z_{ij}\in \spa L$ be elements for $i\in\set{1,\dots,n}$ and $j\in\set{0,\dots,n}$ that satisfy \eqref{lajtorjaformula}; see also Figure~\ref{fig6}. To ease our terminology, the ordered pairs $\pair{y_{ij}}{y_{ij}}$ will be called \emph{witness pairs} (of the containment $\pair{u_1}{v_1} \in \cgi{\spa L}{u_2}{v_2}$). Since $\cgi{\spa L}{u_2}{v_2}\neq \nablaell{\spa L}$, none of the witness pairs generate $\nablaell{\spa L}$. Thus, by \eqref{haiirkok}, 
\begin{equation}
\text{none of the witness pairs is mixed or }1_{\spa H}\text{-colored.}\label{nonm1col}
\end{equation}

Take two consecutive witness pairs, $\pair{y_{i,j-1}}{z_{i,j-1}}$ and $\pair{y_{ij}}{z_{ij}}$.
Here $i,j\in\set{1,\dots,n}$. 
Our next purpose is to show that 
\begin{equation}
\spa\gamma(\pair{y_{i,j-1}}{z_{i,j-1}}) \leqnuspa \spa\gamma(\pair{y_{ij}}{z_{ij}})\text. \label{winprs}
\end{equation}
We assume $y_{i,j-1}<z_{i,j-1}$ since  \eqref{winprs} trivially holds if these two elements are equal. Hence, ${y_{ij}}<{z_{ij}}$ also holds.

\begin{case}[Either  $\pair{y_{i,j-1}}{z_{i,j-1}}$ and $\pair{y_{ij}}{z_{ij}}$ are old, or both are new]  
If both $\pair{y_{i,j-1}}{z_{i,j-1}}$ and $\pair{y_{ij}}{z_{ij}}$ are old pairs, that is, if they belong to $\pairs L$, then \eqref{lajtorjaformula} yields
$ \cgi L{y_{i,j-1}}{z_{i,j-1}}\leq \cgi L{y_{ij}}{z_{ij}}$. From this, we conclude  the relation $\gamma(\pair{y_{i,j-1}}{z_{i,j-1}}) \leqnu \gamma(\pair{y_{ij}}{z_{ij}})$ by \labqb, applied for $\alg L$, and we obtain the validity of \eqref{winprs} for old witness pairs, because  $\spa\gamma$ extends $\gamma$.

If both $\pair{y_{i,j-1}}{z_{i,j-1}}$ and $\pair{y_{ij}}{z_{ij}}$ are new pairs, that is, if they belong to $\pairs{N_6^{pq}}$, then \eqref{lajtorjaformula} and \eqref{nonm1col} allow only two possibilities: $\spa\gamma(\pair{y_{i,j-1}}{z_{i,j-1}}) = \spa\gamma(\pair{y_{ij}}{z_{ij}})$, or $\spa\gamma(\pair{y_{i,j-1}}{z_{i,j-1}}) =p$ and $ \spa\gamma(\pair{y_{ij}}{z_{ij}})=q$. In both cases, \eqref{winprs} holds. 
\end{case}

\begin{case}[$\pair{y_{i,j-1}}{z_{i,j-1}}$ is old and $\pair{y_{ij}}{z_{ij}}$, is new]\label{old2newcase}
Assume first that $j$ is odd, that is,   $\pair{y_{i,j-1}}{z_{i,j-1}}$ is weakly up-perspective into $\pair{y_{ij}}{z_{ij}}$. Since $y_{ij}$, being a new element, and $z_{i,j-1}$ are both distinct from $y_{i,j-1}$, 
\begin{equation}
z_{i,j-1}\parallel y_{ij}\text.\label{osidHn}
\end{equation} 
Since $0_{\spa L}\leq y_{i,j-1}<z_{i,j-1} < z_{ij} $, $z_{i,j-1}$ is an old element, and  $z_{ij}$ is a new one, $z_{ij}\in\set{f_{pq}, g_{pq}}$. Taking $y_{ij}<z_{ij}$ and \eqref{nonm1col} into account, we obtain 
$y_{ij}=f_{pq}$ and $z_{ij}=g_{pq}$. 
Applying the definition of $\leq_{\spa L}$ for the elements of the old witness pair and using the ``weak up-perspectivity relations'' from \eqref{lajtorjaformula}, we have 
$y_{i,j-1}\leq a_p<f_{pq}$. Similarly, but also taking $z_{i,j-1}\parallel y_{ij}$ into account, we obtain $z_{i,j-1}\leq b_p<g_{pq}$. We claim that $\pair{y_{i,j-1}}{z_{i,j-1}}$ is up-perspective to $\pair{a_p}{b_p}$. We can assume $z_{i,j-1}<b_p$, because otherwise  they would be equal, we would have $y_{i,j-1}=z_{i,j-y}\wedge f_{pq}=b_p\wedge f_{pq}=a_p$, and the two pairs would be the same.
Hence, from $a_p\prec b_p$, $z_{i,j-1}<b_p$ and $z_{i,j-1}\parallel y_{ij}=f_{pq}$, we obtain $z_{i,j-1}\parallel a_p$ and $z_{i,j-1}\vee a_p=b_p$. Since $y_{i,j-1}\leq z_{i,j-1}\wedge a_p \leq z_{i,j-1}\wedge y_j=y_{i,j-1}$, the old pair $\pair{y_{i,j-1}}{z_{i,j-1}}$ is up-perspective to the old pair $\pair{a_p}{b_p}$. Hence, $\cgi L{y_{i,j-1}}{z_{i,j-1}}=\cgi L{a_p}{b_p}$. Applying \labqb{} for $\alg L$, we obtain 
\begin{align*}
\spa\gamma(\pair{y_{i,j-1}}{z_{i,j-1}}) 
&=\gamma(\pair{y_{i,j-1}}{z_{i,j-1}}) \overset{\labqb{}}=
\gamma(\pair{a_p}{b_p})
\overset{\labaxc{}}=p\cr 
&=\spa\gamma(\pair{f_{pq}}{f_{pq}})
= \spa\gamma(\pair{y_{ij}}{z_{ij}}),
\end{align*}
which implies \eqref{winprs} if $j$ is odd. 

Second, let $j$ be even. That is, we assume that  $\pair{y_{i,j-1}}{z_{i,j-1}}$ is weakly down-perspective into $\pair{y_{ij}}{z_{ij}}$. The dual of the previous argument shows that $y_{ij}=c_{pq}$ and $z_{ij}\in\set{d_{pq},e_{pq}}$. However, $z_{ij}=d_{pq}$ or $z_{ij}=e_{pq}$ does not make any difference, and 
$\spa\gamma(\pair{y_{i,j-1}}{z_{i,j-1}}) 
=q= \spa\gamma(\pair{a_q}{b_q})= \spa\gamma(\pair{y_{ij}}{z_{ij}})$
settles \eqref{winprs} for $j$ even.
\end{case}

\begin{case}[$\pair{y_{i,j-1}}{z_{i,j-1}}$ is new and  $\pair{y_{ij}}{z_{ij}}$, is old] Like in Case~\ref{old2newcase}, it suffices to deal with an odd  $j$, because an even $j$   could be treated dually. 
Since $\pair{y_{i,j-1}}{z_{i,j-1}}$ is weakly up-perspective into $\pair{y_{ij}}{z_{ij}}$ and $1_{\spa L}$ is the only old element above $f_{pq}$, we obtain $y_{i,j-1}\in\set{c_{pq}, d_{pq},e_{pq}}$.
We obtain  \eqref{osidHn} as before. Taking \eqref{nonm1col} also into account, we obtain that $y_{i,j-1}=c_{pq}$ and $z_{i,j-1}$ is one of $d_{pq}$ and $e_{pq}$. No matter which one, 
an argument dual to the one used in  Case~\ref{old2newcase} yields $a_q=b_{q}\wedge y_{ij}$ and $b_q\leq z_{ij}$. Hence, $\pair{a_q}{b_q}$ is weakly up-perspective into $\pair{y_{ij}}{z_{ij}}$, and we obtain 
\[
\cgi L{a_q}{b_q}\leq \cgi L{y_{ij}}{z_{ij}} \overset{\labqb}{\Then}
q\overset{\labaxc}= \gamma(\pair{a_q}{b_q} )\leqnu \gamma(\pair {y_{ij}}{z_{ij}}),
\]
which implies
\[\spa\gamma(\pair{y_{i,j-1}}{z_{i,j-1}}) =q \leq_{\spa\nu} \spa\gamma(\pair {y_{ij}}{z_{ij}}),
\]
and \eqref{winprs} follows again. 
\end{case}

Now that we have proved \eqref{winprs}, observe that \eqref{winprs} for $j=1,\dots,n$ and transitivity yield
$\spa\gamma(\pair{x_{i-1}}{x_{i}}) =
\spa\gamma(\pair{y_{i0}}{z_{i0}})  \leqnuspa \spa\gamma(\pair{y_{in}}{z_{in}}) = \spa\gamma(\pair{u_2}{v_2})$. Hence, Lemma~\ref{chainlemma} implies
$\spa\gamma(\pair{u_1}{v_1})\leqnuspa \spa\gamma(\pair{u_2}{v_2})$. Therefore, $\spa{\alg L}$ satisfies \labqb{}, and  \labaxa{} holds for $\spa{\alg L}$.

Next, to prove that $\spa{\alg L}$ satisfies \labaxe{}, assume that 
$r,s\in H$ such that  $r\parallel_{\spa\nu} s$ and $\tuple{\delta (r),\epsilon (r), \delta (s),\epsilon (s)}=\tuple{a_r,b_r,a_s,b_s}$   is a spanning $N_6$-quadruple. 
We want to show that it is a strong $N_6$-quadruple of $\spa L$.  The treatment for \eqref{labsb} is almost the dual of that for \eqref{labsa}, whence we give the details only for  \eqref{labsa}. Since the role of $r$ and $s$ is symmetric, it suffices to deal with the case $0<x\leq b_r$; we want to show $x\vee_{\spa L} a_s=1_{\spa L}$.
Since $r\parallel_{\spa\nu} s$  implies $r\parallelnu s$, $L$ is a $\set{0,1}$-sublattice of $\spa L$, and \labaxe{} holds for $\alg L$, we obtain $x\vee_{\spa L} a_s=1_{\spa L}$ for old elements, that is, for all $x\in L$ such that $0<x\leq b_r$. 

Hence, we assume that $x$ is a new element, that is, $x\in N_6^{pq}$. Since $b_r$ is an old element and $x\leq b_r<b_r\vee_L b_s = 1_{\spa L}$, we obtain $x\notin\set{f_{pq},g_{pq}}$. Hence, $x\in  \set{c_{pq},d_{pq},e_{pq}}$. If we had $r\neq q$, then $x\leq b_r$ and the description of $\leq_{\spa L}$ would imply $a_q\leq b_r$, which would be a contradiction since \labaxd{} holds in $\alg L$. Consequently, $r=q$. Thus,  we have $0<x\leq b_q$, and we know from  $s\parallel_{\spa\nu} r=q$ and $p\leq_{\spa \nu} q$  that $s\notin\set{p, q, 0_H}$ and $s\parallelnu q$. We also know $p\neq 0_H$ since $p\parallelnu q$. 

If we had $a_s\in\ideal{g_{pq}}$, then the description of $\leq_{\spa L}$ would yield $a_s\leq b_p$, which would contradict \labaxd{}. 
Hence, $a_s\notin\ideal{g_{pq}}$, and  \eqref{joindeR} gives $x\vee_{\spa L} a_s=\fcs x\vee_L a_s$.  Therefore, since the spanning $N_6$-quadruple $\tuple{a_q,b_q,a_s,b_s}=\tuple{a_r,b_r,a_s,b_s}$ is strong in $\alg L$ by \labaxe{} 
and $0<x<\fcs x\leq b_q$, we conclude $\fcs x\vee_L a_s=1_L$, which implies  the desired  $x\vee_{\spa L} a_s=1_{\spa L}$.  Consequently, $\spa{\alg L}$ satisfies \labaxe{}. This completes the proof of Lemma~\ref{mainlemma}.
\end{proof}

\section{Approaching infinity}
For an ordered set $P=\tuple{P;\leq}$ and a subset $C$ of $P$, the restriction of the ordering of $P$ to $C$ will be denoted by $\restrict{\leq}{C}$. If each element of $P$ has an upper bound in $C$, then $C$ is a \emph{cofinal subset} of $P$. 
The following lemma belongs to the folklore; having no reference at hand, we will outline its easy proof.

\begin{lemma}\label{cofinalitylemma}
If an ordered set  $P=\tuple{P;\leq}$ is the union of a chain of principal  ideals, then it has a cofinal subset $C$ such that $\tuple{C; \restrict{\leq}{C}}$ is a well-ordered set.
\end{lemma}

\begin{proof} The top elements of these principal ideals form a cofinal chain $D$ in $P$. Let $\mathcal H(D)=\{X: X\subseteq D$ and $\tuple{X;\restrict\leq X}$ is a well-ordered set$\}$. For $X,Y\in \mathcal H(D)$, let $X\sqsubseteq Y$ mean that $X$ is an order ideal of $\tuple{Y;\restrict \leq Y}$. Zorn's Lemma yields a maximal member $C$ in $\tuple{\mathcal H(D),\sqsubseteq}$. Clearly, $C$ is well-ordered and it is a cofinal subset.
\end{proof}

Now, we combine the  vertical action of  Lemma~\ref{lupstp}  and the horizontal action of   Lemma~\ref{mainlemma} into a single statement. Note that the order ideal  $H$ of  $\tuple{\vesz H,\vesz \nu}$ in the following lemma is necessarily a directed ordered set.

\begin{lemma}\label{combinlemma} Assume that  $\alg L=\tuple{L;\gamma, H,\nu,\delta ,\epsilon }$ is an auxiliary structure such that $\tuple{H,\nu}$ is an order ideal of a bounded ordered set $\tuple {\vesz H,\vesz \nu}$. $($In particular, $\nu$ is an ordering and $\nu=\restrict{\vesz\nu}{H}$.$)$ Then there exists an auxiliary structure  $\vesz{\alg L}=\tuple{\vesz L;\vesz \gamma, \vesz H,\vesz \nu,\vesz \delta ,\vesz\epsilon }$ such that $\alg L $ is a substructure of $\vesz{\alg L}$. Furthermore, if $\alg L$ and $\vesz H$ are countable, then so is $\vesz{\alg L}$.
\end{lemma}

\begin{proof} We can assume $H\neq \vesz H$ since otherwise $\vesz{\alg L}=\alg L$ would do. 
Consider the set
\begin{equation}
D=\set{\pair pq: 0_{\vesz H}<_{\vesz\nu}  p <_{\vesz\nu} q<_{\vesz\nu} 1_{\vesz H}\text{ and } p\not<_\nu q }\text.\label{ezD}
\end{equation}
 Since every set  can be well-ordered, we can also write 
 $D=\set{\pair{p_\iota}{q_\iota}: \iota<\kappa}$, where $\kappa$ is an ordinal number.  
In $\Quord {\vesz H}$, we define 
\begin{equation}
\nu_\lambda=\bigquos{\nu\cup (\set{0_{\vesz H}}\times \vesz H)
\cup (\vesz H\times \set{1_{\vesz H}})
\cup \set{\pair{p_\iota}{q_\iota}:\iota<\lambda}} \label{sldiHGk}
\end{equation}
for $\lambda\leq \kappa$. It is an  ordering on $\vesz H$, because  $\nu_\lambda\subseteq \vesz \nu$ implies that it is  antisymmetric. Note that $\nu_\kappa=\vesz\nu$ and  $0_{\vesz H}=0_H$. 
For each $\lambda \leq \kappa$, we want to define an auxiliary structure $\alg L_\lambda=\tuple{L_\lambda;\gamma_\lambda, H_\lambda,\nu_\lambda,\delta_{\lambda},\epsilon_{\lambda}}$ 
such that, for all $\lambda<\kappa$, the following properties be satisfied :
\begin{align}
&\text{$\alg L_\mu$ is a substructure of $\alg L_\lambda$ for all $\mu\leq \lambda$;} \label{dirun}  \\
&\text{$H_\lambda=H_0$, $0_{L_\lambda}=0_{L_0}  $, and $1_{L_\lambda}=1_{L_0}$};\label{djrun}  \\
&\begin{aligned} 
\tuple{\delta_\lambda&(p),\epsilon_\lambda(p),\delta_\lambda(q),\epsilon_\lambda(q)}\text{ is a  spanning }N_6\text{-quadruple (equivalently, }  \cr 
&\text{a strong }N_6\text{-quadruple) for all }\pair pq\in D\text{ such that }p\parallel_{\nu_\lambda} q\text.
\end{aligned}\label{spnnning}
\end{align}
Modulo the requirement that $\alg L_\lambda$ should be an auxiliary structure, the equivalence mentioned in \eqref{spnnning} is a consequence of \labaxe{}.  
We define $\alg L_\lambda$ by (transfinite) induction as follows.

\begin{initstep}
We define $\alg L_0$ by a vertical extension.
Let $K=\vesz H\setminus(H\cup\set{1_{\vesz H}} )$, let $\pair{\vcs H}{\vcs\nu}=\pair{\vesz H}{\nu_0}$, and let
$\alg L_0=\vcs{\alg L}$ be the auxiliary structure what we obtain from $\alg L$ according to Lemma~\ref{lupstp}. Note that, for all $\pair pq\in D$,  $\set{p,q}\not\subseteq H$ since $\nu=\restrict{\vesz \nu}{\vesz H}$. Hence, by Lemma~\ref{lupstp}, \eqref{spnnning} holds for $\lambda=0$.
\end{initstep}

\begin{successorstep}
Assume that $\lambda$ is a successor ordinal, that is, $\lambda=\eta+1$, and $\alg L_\eta=\tuple{L_\eta;\gamma_\eta, H_\eta,\nu_\eta,\delta_{\eta},\epsilon_{\eta}}$ is already defined and satisfies \eqref{dirun}, \eqref{djrun}, and \eqref{spnnning}. 
Since $p_\eta <_{\vesz\nu} q_\eta$ and $\nu_\eta\subseteq \vesz \nu$, we have either $p_\eta <_{\nu_\eta} q_\eta$, or $p_\eta \parallel_{\nu_\eta} q_\eta$. These two possibilities need separate treatments. 
First, if $p_\eta <_{\nu_\eta} q_\eta$, then  $\nu_\lambda=\nu_\eta$ and we let $\alg L_\lambda=\alg L_\eta$. 

Second, let $p_\eta \parallel_{\nu_\eta} q_\eta$. We define $\alg L_\lambda$ from $\alg L_\eta$ by a horizontal extension as follows.   With the notation $\spa\nu=\nu_\lambda$, we obtain from \eqref{sldiHGk} that $\spa\nu=\quos{{\nu_\eta\cup\set{\pair{p_\eta}{q_\eta} } }}\in\Quord{\vesz H}$. Furthermore, the validity of 
\eqref{spnnning} for $\alg L_\eta$ yields that $\pair{p_\eta}{q_\eta}$  is a spanning $N_6$-quadruple of   $\alg L_\eta$. Thus, letting $\pair{p_\eta}{q_\eta}$  and $\alg L_\eta$  play the role of $\pair pq$ and $\alg L$ in  \eqref{asumLpar} and \eqref{spaldef}, respectively, we define  ${\alg L_\lambda}$ as the auxiliary structure $\spa{\alg L}$ taken from Lemma~\ref{mainlemma}. Since $L_\eta$ is a $\set{0,1}$-sublattice of $L_\lambda$, spanning $N_6$-quadruples of $L_\eta$ are also spanning in  $L_\lambda$. Furthermore, it follows from $\nu_{\lambda}\supseteq \nu_\eta$ that   $p\parallel_\lambda q\Then p\parallel_\eta q$.  Hence, we conclude that \eqref{spnnning} is inherited by  ${\alg L_\lambda}$ from 
 ${\alg L_\eta}$.
\end{successorstep}

\begin{limitstep} Assume that $\lambda$ is a limit ordinal. Let
\begin{align*}
L_\lambda=\bigcup_{\eta<\lambda} L_\eta,\qq5  \gamma_\lambda=\bigcup_{\eta<\lambda} \gamma_\eta,\qq5 
H_\lambda=\vesz  H,\qq5 
\nu_\lambda=\bigcup_{\eta<\lambda} \nu_\eta,\qq5
\delta_{\lambda }=\bigcup_{\eta<\lambda} \delta_{\eta },
\qq5
\epsilon_{\lambda}=\bigcup_{\eta<\lambda} \epsilon_{\eta }\text.
\end{align*}
We assert that $\alg L_\lambda=\tuple{L_\lambda;\gamma_\lambda, H_\lambda,\nu_\lambda,\delta_{\lambda},\epsilon_{\lambda}}$ is an auxiliary structure satisfying  \eqref{dirun}, \eqref{djrun}, and \eqref{spnnning}. 

Since all the unions defining $\alg L_\lambda$ are directed unions, $L_\lambda$ is a lattice, and $\tuple{H_\lambda;\nu_\lambda}$ is a quasiordered set. Actually, it is an ordered set since $\nu_\lambda\subseteq \vesz \nu$. By the same reason, $\gamma_\lambda$, $\delta_{\lambda}$, and $\epsilon_{\lambda}$ are maps. 
It is straightforward to check that all of 
 \labaxa{},\dots,\labaxg{} hold for $\alg L_\lambda$; we only do this for \labaxa{}, that is, we verify \labqa{} and \labqb{}, and also for \labaxg{}.

Assume $\gamma_\lambda(\pair{u_1}{v_1})\leq_{\nu_\lambda} \gamma_\lambda(\pair{u_2}{v_2})$. Since the unions are directed, there exists an $\eta<\lambda$ such that $u_1,v_1,u_2,v_2\in L_\nu$, and we have $\gamma_\eta(\pair{u_1}{v_1})\leq_{\nu_\eta} \gamma_\eta(\pair{u_2}{v_2})$. Using that the auxiliary structure $\alg L_\eta$ satisfies \labqa{}, we obtain $\cgi{L_\eta}{u_1}{v_1} \leq \cgi{L_\eta}{u_2}{v_2}$, that is,  $\pair{u_1}{v_1} \in \cgi{L_\eta}{u_2}{v_2}$. Using Lemma~\ref{llajtorja}, we conclude $\pair{u_1}{v_1} \in \cgi{L_\lambda}{u_2}{v_2}$ in the usual way. This  implies  $\cgi{L_\lambda}{u_1}{v_1} \leq \cgi{L_\lambda}{u_2}{v_2}$. Therefore, $\alg L_\lambda$ satisfies \labqa{}.

Similarly, if $\cgi{L_\lambda}{u_1}{v_1} \leq \cgi{L_\lambda}{u_2}{v_2}$, then Lemma~\ref{llajtorja} easily implies the existence of an $\eta<\lambda$ such that $\pair{u_1}{v_1} \in \cgi{L_\eta}{u_2}{v_2}$ and
$\cgi{L_\eta}{u_1}{v_1} \leq \cgi{L_\eta}{u_2}{v_2}$; \labqb{} for $\alg L_\eta$ yields $\gamma_\eta(\pair{u_1}{v_1})\leq_{\nu_\eta} \gamma_\eta(\pair{u_2}{v_2})$; and we conclude
$\gamma_\lambda(\pair{u_1}{v_1})\leq_{\nu_\lambda} \gamma_\lambda(\pair{u_2}{v_2})$. Hence, $\alg L_\lambda$ satisfies \labqb{} and \labaxa{}. 

Next, for the sake of contradiction, suppose that \labaxg{} fails in $\alg L_\lambda$. This implies that 
$\pair{0_{L_\lambda}}{1_{L_\lambda}}$ belongs to $\bigvee\bigset{\cgi{L_\lambda}{a_p}{b_p}: p\in H_\lambda\setminus\set{1_{\vesz H}}}$, where the join is taken in the congruence lattice of $L_\lambda$. Since principal congruences are compact, there exists a finite subset $T\subseteq  H_\lambda\setminus\set{1_{\vesz H}}$ such that 
$\pair{0_{L_\lambda}}{1_{L_\lambda}}$ belongs to $\bigvee\set{\cgi{L_\lambda}{a_p}{b_p}: p\in T}$.
Thus, there exists a finite chain $0_{L_\lambda}=c_0<c_1<\dots<c_k=0_{L_\lambda}$ such that, for $i=1,\dots, k$, $\pair{c_{i-1}}{c_i}\in \bigcup\set{\cgi{L_\lambda}{a_p}{b_p}: p\in T}$. Each of these memberships are witnessed by  
finitely many ``witness'' elements according to \eqref{lajtorjaformula}; see Lemma~\ref{llajtorja}. 
Taking all these memberships into account, there are only finitely many witness elements all together. Hence, there exists an $\eta<\lambda$ such that $L_\eta$ contains all these elements. Applying Lemma~\ref{llajtorja} in the converse direction, we obtain that $\pair{0_{L_\eta}}{1_{L_\eta}}=\pair{0_{L_\lambda}}{1_{L_\lambda}}$ belongs to $\bigvee\set{\cgi{L_\eta}{a_p}{b_p}: p\in T}$, which is a contradiction since $\alg L_\eta$ satisfies \labaxg{}.
Consequently, $\alg L_\lambda$ is an auxiliary structure. 

Clearly, $\alg L_\lambda$ satisfies \eqref{dirun} and \eqref{djrun} since so do the $\alg L_\eta$ for $\eta<\lambda$. If $\pair pq\in D$ and $p\parallel_\lambda q$, then $p\parallel_\eta q$ for some (actually, for every) $\eta<\lambda$. 
Hence, the satisfaction of \eqref{spnnning} for $\alg L_\lambda$ follows the same way as in the Successor Step since $L_\eta$ is a $\set{0,1}$-sublattice of $L_\lambda$.
\end{limitstep}

We have seen that $\alg L_\nu$ is an auxiliary structure for all $\lambda\leq\kappa$. Letting $\lambda$ equal $\kappa$, we obtain the existence part of the lemma. The last sentence of the lemma follows from the construction and basic cardinal arithmetics.
\end{proof}

We are now in the position to  complete the paper.

\begin{proof}[Proof of Theorem~\ref{thmmain}] In order to prove part \eqref{thmmainb} of the theorem, assume that  $P=\tuple{P;\nu_P}$ is an ordered set with zero and  it is the union of a chain of principal ideals. By Lemma~\ref{cofinalitylemma}, there exist an ordinal number $\kappa$ and a cofinal chain $C=\set{c_\iota: \iota<\kappa}$ in $P$ such that $0_P=c_0$ and, for $\iota,\mu<\kappa$ we have  $\iota< \mu \iff c_\iota< c_\mu$. The cofinality of $C$ means that  $P$ is the union of the principal ideals $H_\iota=\ideal{c_\iota}$, $\iota<\kappa$. 
We let $H_\kappa=\bigcup_{\iota<\kappa}H_ \iota$ and $\nu_\kappa=\bigcup_{\iota<\kappa}\nu_{H_i}$, where $\nu_{H_i}$ denotes the restriction $\restrict{\nu_P}{H_i}$. 
Clearly, $P=H_\kappa$ and $\nu_P= \nu_\kappa$, that is, $\tuple{P;\nu_P}=\tuple{H_\kappa;\nu_\kappa}$. Note that $H_\kappa$ is not  a principal ideal in general since $P$ need not be bounded.

For each $\lambda \leq \kappa$, we define an auxiliary structure $\alg L_\lambda=\tuple{L_\lambda;\gamma_\lambda, H_\lambda,\nu_\lambda,\delta_{\lambda},\epsilon_{\lambda}}$ 
such that $\alg L_\mu$ is a substructure of $\alg L_\lambda$ for every $\mu\leq \lambda$; we do this by (transfinite) induction as follows.

\begin{initstep} 
We start with the one-element lattice $L_0$ and $H_0=\set{c_0}=\set{0_P}$, and define $\alg L_0$ in the only possible way. 
\end{initstep}

\begin{successorstep}
Assume that  $\lambda=\eta+1$ is a successor ordinal. We apply Lemma~\ref{combinlemma} to obtain $\alg L_\lambda$ from $\alg L_\eta$. This is possible since  $H_\eta$ is an order ideal of $H_\lambda$. Note that Lemma~\ref{combinlemma} does not assert the uniqueness of  $\vesz{\alg L}$, and, in principle, it could be a problem later that  $\alg L_\lambda$ is not uniquely defined. However, this is not a real problem since we can easily solve it as follows.

Let $\tau_0$ be the smallest \emph{infinite} ordinal number such that ${|P|}\leq |\tau_0|$, let $\tau=2^{\tau_0}$,  and let $\pi$ be the smallest ordinal with $|P|=|\pi|$. 
Note that $|\tau|$ is at least the power of continuum but $|\pi|$ can be finite. 
Let $P=\set{h_\iota: \iota<\pi}$ such that $h_\iota\neq h_\eta$ for $\iota<\eta<\pi$. Also,  take a set $T=\set{t_\iota: \iota<\tau}$ such that $t_\iota\neq t_\eta$  for $\iota<\eta<\tau$. The point is that, after selecting the well-ordered cofinal chain $C$ above,  we can use the well-ordered index sets 
$\set{\iota: \iota<\pi}$ and $\set{\iota: \iota<\tau}$ to make every part of our compound construction unique. Namely, when we well-order $D$, defined in \eqref{ezD}, we use the lexicographic ordering of the index set $\set{\iota: \iota<\pi}\times \set{\iota: \iota<\pi}$.
When we define lattices, their base sets will be initial subsets of $T$; a subset $X$ of $T$ is \emph{initial} if, for all $\mu<\iota<\tau$,  $\,t_\iota\in X$  implies $t_\mu\in X$. If we have to add new lattice elements, like a new top or $c_{pq}$, etc., then we always add the first one of $T$ that has not been used previously. 
Cardinality arithmetics shows that $T$ is never exhausted. This way, we have made the definition of $\alg L_\lambda$ unique.

Clearly, $\alg L_\iota$ is a substructure of $\alg L_\lambda$ for $\iota< \lambda$; either by Lemma~\ref{combinlemma} if $\iota=\eta$, or by the induction hypothesis and transitivity if $\iota<\eta$.
\end{successorstep}

\begin{limitstep}
If $\lambda$ is a limit ordinal, then 
first we form the union 
\[\alg L'_\lambda=\tuple{L'_\lambda;\gamma'_\lambda,H'_\lambda,\nu'_\lambda,\delta'_\lambda, \epsilon'_\lambda }=
\tuple{\bigcup_{\eta<\lambda}L_\eta;\bigcup_{\eta<\lambda}\gamma_\eta, \bigcup_{\eta<\lambda}H_\eta,\bigcup_{\eta<\lambda}\nu_\eta,\bigcup_{\eta<\lambda}\delta_\eta, \bigcup_{\eta<\lambda}\epsilon_\eta }\text.
\]
Note that $\nu'_{\lambda}=\restrict {\nu_P} {H'_\lambda}$. The same way as in  the proof of  Lemma~\ref{combinlemma}, we obtain that $\alg L'_\lambda$ is  an auxiliary structure;  the only difference is that now \labaxg{} trivially holds in $\alg L_\lambda$ since $H'_\lambda$ does not have a largest element. To see this, suppose for contradiction that $u$ is the largest element of $H'_\lambda$. Then $u\in H_\eta$ for some $\eta<\lambda$. Since $\lambda$ is a limit ordinal, $\eta+1<\lambda$. Hence $c_{\eta+1}\leq u\leq c_\eta$, which contradicts $c_\eta<c_{\eta+1}$.   

Clearly, $\tuple{H'_\lambda;\nu'_\lambda}$ is an order ideal in $\tuple{H_\lambda;\nu_\lambda}$. Thus, applying Lemma~\ref{combinlemma} to this situation, we obtain an auxiliary structure $\vesz{\alg L}$, and we let $\alg L_\lambda= \vesz{\alg L}$. Obviously, for all $\eta<\lambda$,  $\alg L_\eta$ is a substructure of $\alg L_\lambda$.

Now, we have constructed an auxiliary structure $\alg L_\lambda$ for each $\lambda\leq \kappa$.  In particular, 
$\alg L_\kappa=\tuple{L_\kappa;\gamma_\kappa, H_\kappa,\nu_\kappa,\delta_{\kappa},\epsilon_{\kappa}} =
\tuple{L_\kappa;\gamma_\kappa, P,\nu_P,\delta_{\kappa},\epsilon_{\kappa}} 
$
is an auxiliary structure. Thus, by 
Lemma~\ref{impclM}, $\princ{L_\kappa}\cong\tuple{P;\nu_P}$, which proves  part \eqref{thmmainb} of the theorem. 
\end{limitstep}

In order to prove part \eqref{thmmaina}, 
assume that $L$ is a countable lattice. 
Obviously, we have $|\princ L|\leq|\pairs L|\leq\aleph_0$, and we already mentioned that $\princ L$ is always a directed ordered set with 0, no matter what the size $|L|$ of $L$ is.

Conversely, let $P$ be a directed ordered set with 0 such that $|P|\leq\aleph_0$. Then there is an ordinal $\kappa\leq \omega$ (where $\omega$ denotes the least infinite ordinal) such that $P=\set{p_i: i<\kappa}$. Note that $\set{i:i<\kappa}$ is a subset of the set of nonnegative integer numbers. 
For $i,j<\kappa$, there exists a smallest $k$ such that $p_i\leq p_k$ and $p_j\leq p_k$; we let  $p_i\sqcup p_j=p_k$. This defines a binary  operation on $P$; it need not be a semilattice operation. Let $q_0=p_0$. For $0<i<\kappa$, let $q_i=q_{i-1}\sqcup p_i$. A trivial induction shows that $q_i$ is an upper bound of $\set{p_0,p_1,\dots,p_i}$, for all $i<\kappa$, and $q_{i-1}\leq_P q_i$ for all $0<i<\kappa$.   Hence, 
the principal ideals $\ideal{q_i}$ form a chain 
$\set{\ideal {q_i}: i<\kappa}$, and 
$P$ is the union of these principal ideals. Therefore, part \eqref{thmmainb} of the Theorem yields a lattice $L$ such that $P$ is isomorphic to $\princ L$. Since the $\ideal {q_i}$ are countable and there are countably many of them, and since all the lemmas we used in the proof of part \eqref{thmmainb} of the theorem preserve the property ``countable'', $L$ is countable. 
\end{proof}

\end{document}